\documentclass[11pt]{amsart}

\usepackage{amsfonts}
\usepackage{amssymb}
\usepackage{amsthm}
\usepackage{amsmath}
\usepackage{hyperref}
\usepackage{enumitem}
\usepackage{tikz}
\usepackage{pgf}
\usepackage{graphicx}
\usepackage[top=3.5cm, bottom=3.5cm, left=2.5cm, right=2.5cm]{geometry}

\title{Approximate subgroups of residually nilpotent groups}
\date{}
\author{Matthew C. H. Tointon}
\address{Pembroke College, Cambridge, CB2 1RF, United Kingdom}
\email{mcht2@cam.ac.uk}
\thanks{The author was supported by ERC grant GA617129 `GeTeMo'}

\newtheorem{prop}{Proposition}[section]
\newtheorem{theorem}[prop]{Theorem}
\newtheorem{lemma}[prop]{Lemma}
\newtheorem{corollary}[prop]{Corollary}

\theoremstyle{definition}

\theoremstyle{remark}

\newtheorem{remark}[prop]{Remark}
\newtheorem*{remark*}{Remark}

\theoremstyle{theorem}

\newcommand{\C}{\mathbb{C}}

\newcommand{\N}{\mathbb{N}}
\newcommand{\Z}{\mathbb{Z}}
\newcommand{\Q}{\mathbb{Q}}

\numberwithin{equation}{section}

\begin{document}
\maketitle
\begin{abstract}
We show that a $K$-approximate subgroup $A$ of a residually nilpotent group $G$ is contained in boundedly many cosets of a finite-by-nilpotent subgroup, the nilpotent factor of which is of bounded step. Combined with an earlier result of the author, this implies that $A$ is contained in boundedly many translates of a coset nilprogression of bounded rank and step. The bounds are effective and depend only on $K$; in particular, if $G$ is nilpotent they do not depend on the step of $G$. As an application we show that there is some absolute constant $c$ such that if $G$ is a residually nilpotent group, and if there is an integer $n>1$ such that the ball of radius $n$ in some Cayley graph of $G$ has cardinality bounded by $n^{c\log\log n}$, then $G$ is virtually $(\log n)$-step nilpotent.
\end{abstract}

\tableofcontents

\section{Introduction}
Let $G$ be a group. In recent years there has been a considerable amount of study of subsets $A\subseteq G$ that have \emph{doubling} $K$ in the sense that $|A^2|\le K|A|$, where $K\ge1$ is some parameter. There is much motivation for the study of such sets already in the literature; rather than adding to it here, we simply point out that work in this area has had many applications in an impressively broad range of fields, and that the surveys \cite{bgt.survey,app.grps,ben.icm,helf.survey,sand.survey} provide more detail on the background to the field and on many of these applications.

It turns out that the study of sets in $G$ of bounded doubling essentially reduces to the study of sets called \emph{approximate subgroups} of $G$. A finite set $A\subseteq G$ is said to be a \emph{$K$-approximate subgroup} of $G$ if there exists a set $X\subseteq G$ of size at most $K$ such that $A^2\subseteq XA$. The reader may consult \cite{tao.product.set} for precise details of the relationship between sets of bounded doubling and approximate groups, although it is certainly clear that a $K$-approximate group has doubling at most $K$.

The definition of an approximate group is not particularly descriptive, and a central aim of approximate group theory is to extract more explicit, algebraic information about the structure of approximate groups. The most general result of this type is due to Breuillard, Green and Tao \cite{bgt}, which at its simplest level is as follows.
\begin{theorem}[Breuillard--Green--Tao {\cite[Theorem 2.12]{bgt}}]\label{thm:bgt}
Let $G$ be a group and suppose that $A$ is a $K$-approximate subgroup of $G$. Then there exist subgroups $H\lhd C\subseteq G$ such that
\begin{enumerate}
\item $H\subseteq A^{12}$;
\item $C/H$ is nilpotent of rank and step at most $O(K^2\log K)$;
\item $A$ can be covered by $O_K(1)$ left cosets of $C$.
\end{enumerate}
\end{theorem}
By the \emph{rank} of a nilpotent group here we mean the minimum number of elements needed to generate it as a group.

The use of ultrafilters in the proof of Theorem \ref{thm:bgt} makes it ineffective in the sense that no explicit bound is known on the number of left cosets of $C$ required to cover $A$. There are, however, a number of results of various authors that give explicit bounds in this theorem if one is prepared to restrict to certain specific classes of group; see, for example, \cite{bg,sol.lin,unitary,bgt.lin,bgt.lin.note,freiman,gill-helf,green-ruzsa,helfgott1,helfgott2,pyb-sza,ruzsa.Z,ruzsa,tao.solv,nilp.frei}. The main purpose of this paper is to present a short argument giving explicit bounds in Theorem \ref{thm:bgt} in the case that $G$ is residually nilpotent, as follows.
\begin{theorem}\label{thm:nag}
Let $G$ be a residually nilpotent group and suppose that $A$ is a $K$-approximate subgroup of $G$. Then there exist subgroups $H\lhd C\subseteq G$ such that
\begin{enumerate}
\item $H\subseteq A^{O_K(1)}$;
\item $C/H$ is nilpotent of step at most $K^6$;
\item $A$ can be covered by $\exp(K^{O(1)})$ left cosets of $C$.
\end{enumerate}
\end{theorem}
Let us emphasise in particular that if $G$ is nilpotent then Theorem \ref{thm:nag} applies with bounds that do not depend on the nilpotency class of $G$.

The proof of Theorem \ref{thm:nag} is strongly inspired by the so-called nilpotent Freiman dimension lemma of Breuillard, Green and Tao \cite{nfdl}, which is a similar result valid in the less general setting of a residually \emph{torsion-free} nilpotent group.

\begin{remark*}
It is not known what the optimal bounds should be in Theorems \ref{thm:bgt} and \ref{thm:nag}, although it would be surprising if Theorem \ref{thm:nag} could not be improved. Breuillard and the author \cite[Fact 4.18]{bt} have given an example to show that in Theorem \ref{thm:bgt} one cannot in general cover $A$ with fewer than $K^{\frac{1}{200}\log\log\log\log K}$ cosets of $C$. Eberhard \cite{eberhard} subsequently refined this construction to show that even $K^{\log\log K}$ cosets are not in general sufficient.
\end{remark*}

\begin{remark*}
In principle, our proof of Theorem \ref{thm:nag} also gives a bound on the rank of the nilpotent quotient $C/H$, at least when $G$ is assumed to be nilpotent. However, an earlier result of the author \cite{nilp.frei} gives the much better bound $\exp(\exp(K^{O(1)}))$, as we note in Corollary \ref{cor:nag} below. An explicit bound on the order of the product set of $A$ in which $H$ is contained could also be computed from our argument, but it is rather poor, being roughly a tower of exponentials of height $K^{O(1)}$, and so we do not quantify it precisely.
\end{remark*}

\begin{remark}\label{rem:linear}Every finitely generated linear group is virtually residually nilpotent \cite[Corollary 1.7]{wehrfritz}; see also \cite[pp. 376--377]{lub-seg}. One could therefore, in principle, deduce a version of Theorem \ref{thm:nag} for any given finitely generated subgroup of $GL_n(\mathbb K)$, with $\mathbb K$ an arbitrary field. Care is needed, however. In some cases the consequences are trivial; for example, if $\mathbb{K}$ is a finite field then the arguments of \cite{lub-seg,wehrfritz} exhibit the trivial group as the finite-index residually nilpotent subgroup of $GL_n(\mathbb{K})$. In other cases the consequences are weaker than those given by earlier results; for example, if $\mathbb K$ has characteristic zero and $n$ is fixed then results of Breuillard--Green--Tao \cite{bgt.lin,bgt.lin.note}, or of Pyber--Szab\'o \cite{pyb-sza} and Breuillard--Green \cite{sol.lin}, imply Theorem \ref{thm:nag} for $G\subseteq GL_n(\mathbb K)$ but with stronger bounds. Nonetheless, in certain cases Theorem \ref{thm:nag} does appear to give new information. For example, the kernel of the projection $GL_n(\Z)\to GL_n(\Z/p\Z)$ is residually nilpotent (see \cite[proof of Proposition 1.6]{wehrfritz} or \cite[p. 377]{lub-seg}), and so Theorem \ref{thm:nag} applies directly to its approximate subgroups, whereas the results of \cite{sol.lin,bgt.lin,bgt.lin.note,pyb-sza} do not apply if $n$ is not \emph{a priori} bounded.
\end{remark}

\bigskip

\noindent\textsc{Coset nilprogressions and a more detailed result.} Breuillard, Green and Tao \cite{bgt} in fact proved a more detailed result than that given by Theorem \ref{thm:bgt}. In order to state it we first need a definition. Given elements $x_1,\ldots,x_r\in G$ and positive integers $L_1,\ldots,L_r$, we define the set $P(x_1,\ldots,x_r;L_1,\ldots,L_r)$ to consist of all those elements of $G$ that can be expressed as words in the $x_i$ and their inverses, in which each $x_i$ and its inverse appear at most $L_i$ times between them. If the $x_i$ generate an $s$-step nilpotent subgroup of $G$ then $P(x_1,\ldots,x_r;L_1,\ldots,L_r)$ is said to be a \emph{nilprogression} of \emph{rank} $r$ and \emph{step} $s$. Finally, if $C$ is a subgroup of $G$ and $H$ is a normal subgroup of $C$, and $Q$ is a nilprogression of rank $r$ and step $s$ in $C/H$, then the set $QH$ is said to be a \emph{coset nilprogression} of rank $r$ and step $s$ in $G$.

A more precise version of Theorem \ref{thm:bgt} then states that if $G$ is an arbitrary group, and $A$ is a $K$-approximate subgroup of $G$, then $A$ can be covered by $O_K(1)$ left translates of a coset nilprogression $P\subseteq A^{12}$ of rank and step at most $O(K^2\log K)$ \cite[Theorem 2.12]{bgt}. As before, the bound $O_K(1)$ on the number of left translates of $P$ needed to cover $A$ is ineffective.

The author \cite{nilp.frei} has given an effective version of this more detailed result valid in the case that $G$ is a nilpotent group of bounded step (see Theorem \ref{thm:step}, below). Theorem \ref{thm:nag} allows us to extend this to residually nilpotent groups, and in particular to make the bounds independent of the step of $G$ in the case that $G$ is nilpotent.
\begin{corollary}[Freiman-type theorem for residually nilpotent groups]\label{cor:nag}
Let $G$ be a residually nilpotent group and suppose that $A$ is a $K$-approximate subgroup of $G$. Then there is a coset nilprogression $P\subseteq A^{O_K(1)}$ of rank at most $\exp(\exp(K^{O(1)}))$ and step at most $K^6$ such that $A$ can be covered by $\exp(K^{O(1)})$ left translates of $P$.
\end{corollary}

\begin{remark*}In the case that $G$ is abelian the so-called \emph{polynomial Freiman--Ruzsa conjecture} asserts that a $K$-approximate group $A$ should be covered by $K^{O(1)}$ translates of a coset progression of rank $O(\log K)$ and cardinality at most $|A|$. These bounds would be optimal. The best result in this direction is due to Sanders \cite{sanders}, who has shown that one can cover $A$ with $\exp((\log K)^{O(1)})$ translates of a progression of rank $(\log K)^{O(1)}$. Schoen \cite{schoen} had previously obtained essentially optimal bounds in a closely related variant of this problem, showing that $A$ is contained in a \emph{single} coset progression of dimension at most $K^{1+O((\log K)^{-1/2})}$ and cardinality at most $\exp(K^{1+O((\log K)^{-1/2})})|A|$ (similar bounds in this variant can also be computed from the Sanders result).
\end{remark*}

\begin{remark*}
The abelian case of Corollary \ref{cor:nag} (stated as Theorem \ref{thm:ab.frei}, below) is ultimately an ingredient in the proof of Corollary \ref{cor:nag}. It appears that if one modified the argument of \cite{nilp.frei} to optimise the rank of the nilprogression rather than the number of translates required to cover the approximate group, and then applied the arguments of the present paper with the Sanders bounds in Theorem \ref{thm:ab.frei}, one could swap the bounds on the rank and the number of translates in Corollary \ref{cor:nag} (thus, the rank of $P$ would be at most $\exp(K^{O(1)})$, with $A$ being covered by $\exp(\exp(K^{O(1)}))$ translates of $P$). However, it does not appear that one could improve either bound without worsening the other, even assuming the polynomial Freiman--Ruzsa conjecture, and so we do not pursue this matter here.
\end{remark*}

\begin{remark*}
Breuillard, Green and Tao's more detailed version \cite[Theorem 2.12]{bgt} of Theorem \ref{thm:bgt} actually gives a bit more qualitative information than Corollary \ref{cor:nag}. Specifically, the coset nilprogression can be taken to be in \emph{$O_K(1)$-normal form} (see \cite[Definition 2.6]{bgt}). In a very recent preprint \cite{tess-toi}, Tessera and the author have shown that the coset nilprogression of Theorem \ref{thm:step}, and hence that of Corollary \ref{cor:nag}, can also be taken to be in $O_K(1)$-normal form. See \cite{tess-toi} for more details and background.
\end{remark*}

\bigskip

\noindent\textsc{Residually nilpotent groups of bounded exponent.} Let us point out a specific setting in which our argument gives stronger bounds than those of Theorem \ref{thm:nag} and Corollary \ref{cor:nag}. Ruzsa \cite{ruzsa} famously showed that if $A$ is a set of doubling $K$ inside an abelian group in which every element has order at most $r$, then $A$ is contained inside a genuine subgroup of cardinality at most $r^{K^4}K^2|A|$. Our argument provides the following generalisation of this statement to residually nilpotent groups.
\begin{theorem}\label{thm:pgrp}
Let $G$ be a residually nilpotent group in which every element has order at most $r$. Suppose that $A$ is a $K$-approximate subgroup of $G$. Then $A$ can be covered by $K^{35K^6+2}$ left cosets of a nilpotent subgroup contained in $A^{(3r+2)K^6+2}$.
\end{theorem}

\bigskip

\noindent\textsc{A one-scale growth gap for residually nilpotent groups.} Let $G$ be a group with finite symmetric generating set $S$. A well-known and remarkable theorem of Gromov \cite{gromov} states that if $|S^n|$ is bounded by some polynomial in $n$ then $G$ is virtually nilpotent. There have since been various refinements and strengthenings of this result. Some, such as \cite[Theorem 7.1]{hrush}, \cite[Corollaries 11.2, 11.5 \& 11.7]{bgt} and \cite[Theorem 4.1]{bt}, were proved using approximate groups; in particular, each of these follows from Theorem \ref{thm:bgt} or variants of it. As one might therefore expect, Theorem \ref{thm:nag} also yields a refinement of Gromov's theorem in the residually nilpotent case.

Before we present this result, let us note that Shalom and Tao \cite{shalom-tao} have already given a refinement of Gromov's theorem in the general case, showing that there exists $c>0$ such that if
\begin{equation}\label{eq:shal.tao}
|S^n|\le n^{(\log\log n)^c},
\end{equation}
for some $n>1$ then $G$ is virtually nilpotent. In the residually nilpotent case, Grigorchuk and Lubotzky and Mann have shown that one can weaken the bound required on $|S^n|$ yet further: they show that if
\begin{equation}\label{eq:gap}
|S^n|< 2^{\lfloor\sqrt{n}\rfloor}
\end{equation}
for infinitely many $n\in\N$ then $G$ is virtually nilpotent \cite[Theorem E2]{anal.pro.p}. Grigorchuk \cite{grigorchuk} proved this first in the case that $G$ is residually a $p$-group; a lemma of Lubotzky and Mann \cite[Lemma 1.7]{lub-mann} then shows that his argument still works under the weaker assumption that $G$ is residually nilpotent. It has been suggested that \eqref{eq:gap} could be enough to imply that an \emph{arbitrary} group is virtually nilpotent \cite{gap.conj}.

Note that whereas the Grigorchuk--Lubotzky--Mann result requires the bound \eqref{eq:gap} to hold for infinitely many $n$ (we call this a `multi-scale' hypothesis), the Shalom--Tao result requires only that the bound \eqref{eq:shal.tao} hold for a single value of $n$ (we call this a `one-scale' hypothesis). It is not known whether the bound \eqref{eq:shal.tao} can be weakened further at the expense of reverting to a multi-scale hypothesis.

The following corollary of Theorem \ref{thm:nag} shows that in the class of residually nilpotent groups one has Gromov's theorem under a one-scale hypothesis with a slightly weaker bound than \eqref{eq:shal.tao}, and goes via a completely different argument to those of Grigorchuk, Lubotzky--Mann and Shalom--Tao.

\begin{corollary}[one-scale growth gap for residually nilpotent groups]\label{cor:growth}
There exists an absolute constant $c>0$ such that if $G$ is a residually nilpotent group with finite symmetric generating set $S$, and if there exists $n>1$ such that 
\begin{equation}\label{eq:1-sc.gap}
|S^n|\le n^{c\log\log n},
\end{equation}
then $G$ contains a $(\log n)$-step nilpotent subgroup of index $O_n(1)$.
\end{corollary}

\begin{remark*}
As in Remark \ref{rem:linear}, Corollary \ref{cor:growth} implies a growth-gap result for linear groups. Specifically, Corollary \ref{cor:growth} holds with the same constant $c$ when $G$ is a linear group, provided $n$ is large enough in terms of the dimension of $G$ and the ring generated by the matrix entries of a generating set for $G$. However, much stronger results should be available using the uniform Tits alternative (see the papers \cite{br.height.gap,br.tits} of Breuillard and \cite{br-gel} of Breuillard--Gelander) and uniform exponential growth for soluble groups (see the paper \cite{br.growth.sol} of Breuillard), and so we omit the details.
\end{remark*}

\bigskip\noindent\textsc{Outline of the paper.} In Section \ref{sec:app.grps} we review the necessary background material on approximate groups. In Section \ref{sec:prelim} we prove a preliminary structure theorem for nilpotent approximate groups, which is essentially the argument of \cite{nfdl} adapted to deal with the possibility of finite subgroups. We also deduce Theorem \ref{thm:pgrp} in the specific case that $G$ is nilpotent. In Section \ref{sec:central} we prove a structure theorem for an approximate subgroup of a nilpotent group $G$ that surjects onto the quotient $G/Z(G)$, and then in Section \ref{sec:conclusion} we combine everything to prove Theorem \ref{thm:nag} in the case that $G$ is nilpotent. In Section \ref{sec:resid} we deduce the general statements of Theorems \ref{thm:nag} and \ref{thm:pgrp} from their respective nilpotent versions, as well as proving Corollary \ref{cor:nag}. Finally, in Section \ref{sec:growth} we prove Corollary \ref{cor:growth}

\bigskip\noindent\textsc{Acknowledgements.} It is a pleasure to thank Emmanuel Breuillard for valuable conversations, comments and corrections; Corina Ciobotaru for corrections; Tom Sanders for help with the references; Terence Tao for encouraging the pursuit of this problem; and an anonymous referee for a careful reading of the paper and a number of helpful comments.

\section{Background on approximate groups}\label{sec:app.grps}
In this section we collect together various basic facts about approximate groups. We start with a simple but powerful combinatorial lemma, based on an earlier result of Gleason \cite[Lemma 1]{gleason}. This is a key tool in the nilpotent Freiman dimension argument of Breuillard, Green and Tao \cite{nfdl}, where it essentially allows the authors to bound the dimension of a torsion-free nilpotent $K$-approximate group in terms of $K$. Since the dimension also bounds the step, this is sufficient to imply Theorem \ref{thm:nag} in this case.
\begin{lemma}\label{lem:gleason}
Let $A$ be a finite symmetric subset of a group and let $m\in\N$. Let $\{1\}=H_0\subseteq H_1\subseteq\ldots\subseteq H_k$ be a nested sequence of groups such that $A^m\cap H_i\not\subseteq A^2H_{i-1}$. Then $|A^{m+1}|\ge k|A|$.
\end{lemma}
\begin{proof}
This is essentially \cite[Lemma 3.1]{nfdl}. For each $i=1,\ldots,k$ pick $h_i\in(A^m\cap H_i)\backslash A^2H_{i-1}$. It is sufficient to show that the sets $Ah_i$ are all disjoint. To see this, suppose that $Ah_i\cap Ah_j\ne\varnothing$ for some $j<i$. This would imply that $h_i\in A^2h_j\subseteq A^2H_j\subseteq A^2H_{i-1}$, contradicting the choice of $h_i$.
\end{proof}

The following standard lemma may be found in \cite[Lemma 2.10]{nilp.frei}, for example.
\begin{lemma}\label{lem:int.app.grp}
Let $G$ be a group with a subgroup $H$, and suppose that $A$ is a $K$-approximate subgroup of $G$. Let $m\in\N$. Then $A^m\cap H$ can be covered by $K^{m-1}$ left translates of $A^2\cap H$. In particular, $A^m\cap H$ is a $K^{2m-1}$-approximate group for every $m\ge2$.
\end{lemma}

\begin{lemma}\label{lem:2.2.ii}
Let $G$ be a group and $H$ a subgroup. Suppose that $A$ is a finite symmetric subset of $G$. Then $|A|\le|A^2\cap H||AH/H|\le|A^3|$. Moreover, $A$ is covered by $|AH/H|$ translates of $A^2\cap H$.
\end{lemma}
\begin{proof}
This is essentially {\cite[Lemma 2.2 (i)]{nfdl}}. Let $X\subseteq A$ be a minimal set of left-coset representatives for $H$ in $AH$, so that $|X|=|AH/H|$, and note that $|X(A^2\cap H)|=|X||A^2\cap H|$ . The lemma then follows from the fact that $A\subseteq X(A^2\cap H)\subseteq A^3$.
\end{proof}

\begin{corollary}\label{cor:ruz.cov}
Let $G$ be a group, let $A$ be a $K$-approximate subgroup, and let $H$ be a subgroup such that $|A^2\cap H|\ge|A|/K'$. Then $|AH/H|\le K'K^2$.
\end{corollary}
\begin{proof}
This observation is made in the proof of \cite[Theorem 1.1]{nfdl}. The upper bound of Lemma \ref{lem:2.2.ii}, the approximate group property and the hypothesis of the corollary imply that $|A^2\cap H||AH/H|\le K^2|A|\le K^2K'|A^2\cap H|$.
\end{proof}

\begin{lemma}\label{lem:size.of.int}
Let $A$ be a finite symmetric set in a group, and let $H$ be a subgroup such that $A^2\cap H=\{1\}$. Then $|A^m\cap H|\le |A^{m+1}|/|A|$.
\end{lemma}
\begin{proof}
This is essentially found in the proof of \cite[Proposition 4.1]{nfdl}. First note that the sets $aH$ with $a\in A$ are disjoint. Indeed, if $a,a'\in A$ and $aH\cap a'H\ne\varnothing$ then $a^{-1}a'\in A^2\cap H$, and so $a=a'$. This implies in particular that $|A(A^m\cap H)|=|A||A^m\cap H|$, and the lemma follows.
\end{proof}

The next result is another key lemma from the Breuillard--Green--Tao nilpotent Freiman dimension argument \cite{nfdl}, where it allows the authors to locate an element in an approximate group with a large centraliser, which is in turn a key ingredient in finding a large nilpotent piece of that approximate group. It is also somewhat reminscent of \cite[Proposition 4.1]{helfgott1}.

\begin{lemma}\label{lem:centraliser}
Let $A$ be a $K$-approximate subgroup of a group $G$, and let $G=Z_1\supseteq Z_2\supseteq\ldots$ be a central series for $G$. Let $j$ be maximal such that $A^2\cap Z_j\ne\{1\}$, and let $\omega\in A^m\cap Z_j$. Then
\[
|A^2\cap C_G(\omega)|\ge\frac{|A|}{K^{2m+2}}.
\]
\end{lemma}
\begin{proof}
This is essentially found in the proof of \cite[Proposition 4.1]{nfdl}. For each $a\in A$ we have $[\omega,a]\in A^{2m+2}\cap Z_{j+1}$. Lemma \ref{lem:size.of.int} and the definition of $j$ imply that $|A^{2m+2}\cap Z_{j+1}|\le K^{2m+2}$, and so as $a$ ranges through $A$ the number of values taken by $[\omega,a]$ is at most $K^{2m+2}$. Fix $a$ so that $[\omega,a]$ is the most popular such value, so that $[\omega,x]=[\omega,a]$ for at least $|A|/K^{2m+2}$ elements $x\in A$. For each such $x$ we have $xa^{-1}\in A^2\cap C_G(\omega)$, and so the lemma holds.
\end{proof}

\begin{theorem}[Green--Ruzsa \cite{green-ruzsa}]\label{thm:ab.frei}
Suppose that $A$ is a $K$-approximte subgroup of an abelian group. Then there exist a subgroup $H\subseteq4A$, and elements $x_1,\ldots,x_r\in4A$ and positive integers $L_1,\ldots,L_r$ with $r\le K^{O(1)}$, such that $A\subseteq H+P(x_1,\ldots,x_r;L_1,\ldots,L_r)\subseteq K^{O(1)}A$.
\end{theorem}

As we remarked in the introduction, Sanders \cite{sanders} has shown that one can take the rank of $P$ in Theorem \ref{thm:ab.frei} to be $(\log K)^{O(1)}$, with $A$ now contained in $\exp((\log K)^{O(1)})$ translates of $H+P$. It does not appear that this leads to better bounds in Theorem \ref{thm:nag}.

\begin{theorem}\label{thm:step}
Let $G$ be an $s$-step nilpotent group and suppose that $A$ is a $K$-approximate subgroup of $G$. Then there exists a coset nilprogression $P$ of rank at most $K^{e^{O(s)}}$ such that $A\subseteq P\subseteq A^{K^{O_s(1)}}$.
\end{theorem}
\begin{proof}This is \cite[Theorem 1.5]{nilp.frei}, except that the bound on the rank of $P$ stated in \cite{nilp.frei} is $K^{O_s(1)}$. The more precise $K^{e^{O(s)}}$ claimed here follows from an inspection of the proof in \cite{nilp.frei}.
\end{proof}

\section{A preliminary structure theorem for nilpotent approximate groups}\label{sec:prelim}
The strategy of Breuillard, Green and Tao's nilpotent Freiman dimension argument \cite{nfdl} is roughly as follows. Given a $K$-approximate subgroup $A$ of a torsion-free nilpotent group $G$, they seek a large piece of $A$ that is nilpotent of bounded step. They first use Lemma \ref{lem:centraliser} to locate an element $\gamma_1\in A$ with a large centraliser; passing to a group of bounded index, they can in fact assume that $\gamma_1$ is central. Writing $H_1$ for the largest cyclic subgroup containing $\gamma_1$, they then pass to the quotient $G/H_1$, which is automatically torsion-free, and repeat, producing a sequence $\gamma_1,\gamma_2,\ldots,\gamma_k$.

Writing $H_i=\langle\gamma_1,\ldots,\gamma_i\rangle$, since $G/H_{i-1}$ is torsion-free and $A$ is finite, each $\gamma_i$ has a power that is not contained in $A^2$ modulo $H_i$. This element $\gamma_i$ therefore contributes to the doubling of $A$ in the sense of Lemma \ref{lem:gleason}, and so that lemma implies that the number of elements $\gamma_i$ this process produces is bounded in terms of $K$. In particular, this process gives a central series of bounded length that contains a large piece of $A$, and this piece is therefore of bounded step.

In the setting of the present paper, the fact that $G$ may have torsion means we cannot assume in the same way that $\gamma_i$ contributes to the doubling of $A$. Indeed, it is possible that $\gamma_i$ generates a subgroup that is entirely contained in $A$ modulo $H_{i-1}$, and hence makes no contribution to the doubling of $A$ in the sense of Lemma \ref{lem:gleason}. We must therefore content ourselves with the following statement.
\begin{prop}\label{prop:dim.lem}
Let $G$ be a nilpotent group and let $A$ be a $K$-approximate subgroup of $G$. Then there exist $k\le K^6$, subgroups $D_1\subseteq\ldots\subseteq D_k\subseteq D_{k+1}$ and $\langle A\rangle=C_0\supseteq C_1\supseteq\ldots\supseteq C_k$ of $G$, and elements $\gamma_1,\ldots,\gamma_k$ such that $\gamma_i$ normalises $D_i$ and such that, writing $H_0=\{1\}$ and $H_i=\langle\gamma_i\rangle D_i$ for $i=1,\ldots,k$, we have that
\begin{enumerate}
\item $D_i\lhd C_{i-1}$;\label{item:dc}
\item $\gamma_i$ is central in $C_i/D_i$; in particular $H_i\lhd C_i$;\label{item:h}
\item $H_{i-1}\subseteq D_i$;\label{item:h.nest}
\item $D_i\subseteq A^2H_{i-1}$;\label{item:dah}
\item $C_i=\langle A^2\cap C_i\rangle H_i$;\label{item:c.gen}
\item $\gamma_i\in A^6\backslash A^2H_{i-1}$;\label{item:gam2}
\item $|A^2\cap C_i|\ge K^{-35i}|A|$;\label{item:size}
\item $C_k=D_{k+1}$.\label{item:cah}
\end{enumerate}
\end{prop}

\begin{figure}[t]
\begin{center}
\begin{tikzpicture}

\draw (0,0) node {$D_{k+1}$};
\draw (1,0) node {$=$};
\draw (2,0) node {$C_k$};
\draw (3,0) node {$\subseteq$};
\draw (4,0) node {$C_{k-1}$};
\draw (5,0) node {$\subseteq$};
\draw (6,0) node {$\cdots$};
\draw (7,0) node {$\subseteq$};
\draw (8,0) node {$C_1$};
\draw (9,0) node {$\subseteq$};
\draw (10.5,0) node {$C_0=\langle A\rangle$};

\draw (0,-2) node {$D_k$};
\draw (0,-4) node {$D_{k-1}$};
\draw (0,-5.9) node {$\vdots$};
\draw (0,-8) node {$D_2$};
\draw (0,-10) node {$D_1$};

\draw (0,-1) node {$H_k$};
\draw (0,-3) node {$H_{k-1}$};
\draw (0,-9) node {$H_1$};
\draw (0,-11) node {$H_0=\{1\}$};

\draw (0,-0.3) -- (0,-0.7);
\draw (0,-1.3) -- (0,-1.7);
\draw (0,-2.3) -- (0,-2.7);
\draw (0,-3.3) -- (0,-3.7);
\draw (0,-4.3) -- (0,-5.6);
\draw (0,-6.4) -- (0,-7.7);
\draw (0,-8.3) -- (0,-8.7);
\draw (0,-9.3) -- (0,-9.7);
\draw (0,-10.3) -- (0,-10.7);

\draw (0.25,-0.85) -- node[below, rotate=35] {$\lhd$} (1.7,-0.15);
\draw (0.25,-2.85) -- node[below right, rotate=40] {$\lhd$} (3.6,-0.15);
\draw (0.25,-8.85) -- node[below right, rotate=45] {$\lhd$} (7.8,-0.15);
\draw (0.2,-10.7) -- node[below right, rotate=45] {$\lhd$} (9.9,-0.15);

\draw (0.25,-1.85) -- node[above left, rotate=35] {$\lhd$} (3.6,-0.15);
\draw (0.25,-7.85) -- node[above left, rotate=45] {$\lhd$} (7.8,-0.15);
\draw (0.25,-9.7) -- node[above left, rotate=45] {$\lhd$} (9.9,-0.15);

\end{tikzpicture}
\caption{Illustration of Proposition \ref{prop:dim.lem}}\label{fig:dim.lem}
\end{center}
\end{figure}
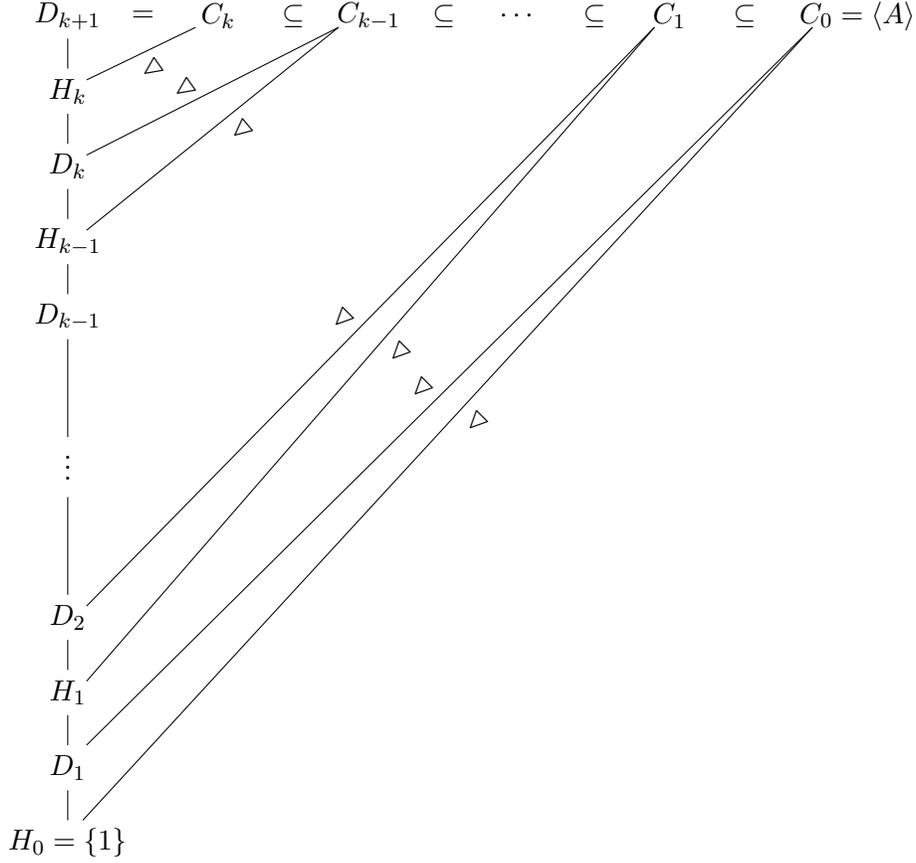

The subgroups and inclusions given by this proposition are illustrated in Figure \ref{fig:dim.lem}. The key output to note is the normal series
\begin{equation}\label{eq:chain}
\{1\}=H_0\lhd D_1\lhd H_1\lhd D_2\lhd\ldots\lhd D_k\lhd H_k\lhd D_{k+1}
\end{equation}
(that this series is normal is not stated explicitly in the proposition, but follows immediately from it). Each subgroup $H_i$ is cyclic and central modulo $D_i$, being generated by the element $\gamma_i$ modulo $D_i$. The groups $H_i$ and the elements $\gamma_i$ are analogous to the groups $H_i$ and elements $\gamma_i$ in the description above of the argument of Breuillard, Green and Tao.

The groups $D_i$, on the other hand, do not feature in the torsion-free setting of Breuillard, Green and Tao; they correspond to the elements, described before Proposition \ref{prop:dim.lem}, that do not contribute to the doubling of $A$ in the torsion setting (note that conclusion \eqref{item:dah} of the proposition implies that $D_i$ is contained in $A^2$ modulo $H_{i-1}$; in particular, $D_i/H_{i-1}$ is a finite group).

We will not ultimately be interested in the subgroups $C_i$, but their presence in the statement of the proposition makes it easier to formulate our inductive proof. The reader may therefore ignore these groups when it comes to applying Proposition \ref{prop:dim.lem} in later sections.

We start our proof of Proposition \ref{prop:dim.lem} with the following lemma.

\begin{lemma}\label{lem:normal.int}
Let $G$ be a group with symmetric generating set $B$, and let $Z$ be a normal subgroup of $G$. Suppose that $B\cap Z$ is not a normal subgroup of $G$. Then $B^3\cap Z\backslash B\ne\varnothing$.
\end{lemma}
\begin{proof}
If $B\cap Z$ is not a subgroup then we have the stronger statement that $B^2\cap Z\backslash B\ne\varnothing$. If $B\cap Z$ is a non-normal subgroup then there exist $b\in B$ and $x\in B\cap Z$ such that $b^{-1}xb\notin B\cap Z$. Since $b^{-1}xb\in B^3\cap Z$, the lemma is proved.
\end{proof}

\begin{proof}[Proof of Proposition \ref{prop:dim.lem}]
Noting that $C_0=\langle A\rangle$ and $H_0=\{1\}$ always satisfiy conditions \eqref{item:h}, \eqref{item:c.gen} and \eqref{item:size} of the proposition for $i=0$, we show that if subgroups $C_0,\ldots,C_j$ and $D_1,\ldots,D_j$ and elements $\gamma_1,\ldots,\gamma_j$ exist and satisfy the first seven conditions of the proposition for $i=0,\ldots,j$, and if $C_j\not\subseteq A^2H_j$, then we can construct $C_{j+1},D_{j+1},\gamma_{j+1}$ that satisfy the first seven conditions for $i=j+1$. On the other hand, if we do have $C_j\subseteq A^2H_j$ then we stop and take $k=j$ and $D_{j+1}=C_j$ so that conditions \eqref{item:dc}, \eqref{item:h.nest}, \eqref{item:dah} and \eqref{item:cah} are satisfied. Note that if $j>K^6$ then property \eqref{item:gam2} would contradict Lemma \ref{lem:gleason}, so the process must terminate with $k\le K^6$.

Write $\pi_j$ for the projection homomorphism $\pi_j:C_j\to C_j/H_j$, and set $B_j=\pi_j(A^2\cap C_j)$. Let $C_j/H_j=Z_1\supseteq Z_2\supseteq\ldots$ be a central series for $C_j/H_j$. If $B_j=C_j/H_j$ then we have $C_j\subseteq A^2H_j$, and the process stops; in particular, if $C_j\not\subseteq A^2H_j$ then $B_j^2\ne B_j$ by property \eqref{item:c.gen}, and so we may set $\ell$ to be maximal such that there exists $\omega_j\in B_j^3\cap Z_\ell\backslash B_j$. Property \eqref{item:c.gen} for $i=j$, along with Lemma \ref{lem:normal.int} and the maximality of $\ell$, imply that $B_j\cap Z_{\ell+1}$ is a normal subgroup of $C_j/H_j$, and so we may set $D_{j+1}$ to be its pullback $\pi_j^{-1}(B_j\cap Z_{\ell+1})$, which satisfies conditions \eqref{item:dc}, \eqref{item:h.nest} and \eqref{item:dah} for $i=j+1$.

Now write $\rho_j$ for the projection homomorphism $\rho_j:C_j/H_j\to C_j/D_{j+1}$. Note that $\rho_j(\omega_j)\ne1$, and moreover that this implies that $\rho_j(B_j)\ne\{1\}$ (since $\rho_j(B_j)$ generates $C_j/D_{j+1}$ by the $i=j$ case of property \eqref{item:c.gen}). We may therefore let $n$ be maximal such that $\rho_j(B_j^2)\cap\rho_j(Z_n)\ne\{1\}$. This implies that there exist $b\in B_j^2$ and $z\in Z_n$ such that $\rho_j(b)=\rho_j(z)\ne1$, and in particular that there exists $h\in B_j\cap Z_{\ell+1}$ such that $z=bh$.

We conclude that $z\in B_j^3\cap Z_n$, and so since $\rho_j(z)\ne1$ we have $z\in (B_j^3\cap Z_n)\backslash(B_j\cap Z_{\ell+1})$. Thus $n\le\ell$ by definition of $\ell$. In particular, $\rho_j(\omega_j)\in\rho_j(B_j^3)\cap\rho_j(Z_n)$, and so Lemmas \ref{lem:int.app.grp} and \ref{lem:centraliser} imply that
\[
|\rho_j(B_j^2)\cap C_{C_j/D_{j+1}}(\rho_j(\omega_j))|\ge\frac{|\rho_j(B_j)|}{K^{24}}.
\]
Defining $G_{j+1}=\pi_j^{-1}\circ\rho_j^{-1}(C_{C_j/D_{j+1}}(\rho_j(\omega_j)))$, we therefore have
\begin{equation}\label{eq:size.1}
|\rho_j\circ\pi_j(A^4\cap G_{j+1})|\ge\frac{|\rho_j(B_j)|}{K^{24}}.
\end{equation}
Moreover, following \cite{nfdl}, we have
\begin{align*}
K^{11}|A^2\cap G_{j+1}| &\ge |A^{12}\cap G_{j+1}|   &\text{(by Lemma \ref{lem:int.app.grp})}\\
    &\ge |(A^4\cap G_{j+1})^3|\\ 
    &\ge |(A^4\cap G_{j+1})^2\cap D_{j+1}||\rho_j\circ\pi_j(A^4\cap G_{j+1})|   &\text{(by Lemma \ref{lem:2.2.ii})}.
\end{align*}
Since $(A^4\cap G_{j+1})^2\cap D_{j+1}\supseteq A^4\cap D_{j+1}$, we conclude that
\[
|A^2\cap G_{j+1}|\ge K^{-11}|A^4\cap D_{j+1}||\rho_j\circ\pi_j(A^4\cap G_{j+1})|,
\]
which combines with \eqref{eq:size.1} to imply that
\begin{align*}
|A^2\cap G_{j+1}| &\ge K^{-35}|A^4\cap D_{j+1}||\rho_j(B_j)|\\
    &=K^{-35}|A^4\cap D_{j+1}||\rho_j\circ\pi_j(A^2\cap C_j)|\\
    &\ge K^{-35}|(A^2\cap C_j)^2\cap D_{j+1}||\rho_j\circ\pi_j(A^2\cap C_j)|\\
    &\ge K^{-35}|A^2\cap C_j| &\text{(by Lemma \ref{lem:2.2.ii})}\\
    &\ge K^{-35(j+1)}|A| &\text{(by property \eqref{item:size} for $i=j$).}
\end{align*}
Pick an arbitrary element $\gamma_{j+1}\in A^6\cap C_j$ such that $\pi_j(\gamma_{j+1})=\omega_j$, and note that $\gamma_{j+1}$ satisfies \eqref{item:gam2} for $i=j+1$ and, being contained in $C_j$, normalises $D_{j+1}$. Finally, define $C_{j+1}=\langle A^2\cap G_{j+1}\rangle H_{j+1}$, noting that this satisfies \eqref{item:c.gen} and \eqref{item:size} for $i=j+1$. Moreover, (the image of) $\gamma_{j+1}$ is central in $G_{j+1}/D_{j+i}$ by definition, and so in particular it is central in $C_{j+1}/D_{j+i}$, and so \eqref{item:h} is satisfied for $i=j+1$.
\end{proof}

\begin{proof}[Proof of Theorem \ref{thm:pgrp} (nilpotent case)]
Apply Proposition \ref{prop:dim.lem}, noting that property \eqref{item:gam2} implies that $\gamma_i\in A^6$, and hence that $\langle\gamma_i\rangle\subseteq A^{3r}$. It then follows from repeated application of property \eqref{item:dah} that $H_k\subseteq A^{(3r+2)K^6}$, and hence from property \eqref{item:cah} that $C_k\subseteq A^{(3r+2)K^6+2}$. The result then follows from property \eqref{item:size} for $i=k$, and Corollary \ref{cor:ruz.cov}.
\end{proof}

\section{Central extensions of nilpotent approximate groups}\label{sec:central}
Theorem \ref{thm:nag} requires us to exhibit a group $H\subseteq A^{O_K(1)}$ and a group $C$ such that $H\lhd C$ and $C/H$ is nilpotent of bounded step. So far, we have succeeded only in producing the chain \eqref{eq:chain} given by Proposition \ref{prop:dim.lem}, which in fact consists of several quotients $D_i/D_{i-1}$ that are in some sense the opposite of what we are looking for: $H_{i-1}$ is nilpotent of bounded step (indeed, abelian) in the quotient $D_i/D_{i-1}$, whilst $D_i/H_{i-1}$ is finite.

We have $(A^2\cap D_i)H_{i-1}=D_i$ by conclusions \eqref{item:h.nest} and \eqref{item:dah} of Proposition \ref{prop:dim.lem}, and the group $H_{i-1}$ is central in the quotient $D_i/D_{i-1}$ by conclusion \eqref{item:h}. The quotient $D_i/D_{i-1}$ may therefore be thought of as a `central extension' of the approximate group $A^2\cap D_i$. In this section we describe the structure of such central extensions of nilpotent approximate groups, as follows.
\begin{prop}\label{prop:strong.dim.lem}
Let $G$ be a finitely generated nilpotent group, and let $A$ be a $K$-approximate subgroup such that $G=A\cdot Z(G)$. Then there exist $k\le K^8$, and normal subgroups $\{1\}=H_0\subseteq H_1\subseteq\ldots\subseteq H_k\subseteq[G,G]$ of $G$ such that $H_i\subseteq A^8H_{i-1}$, and such that $[G,G]\subseteq A^4H_k$. In particular, $[G,G]\subseteq A^{8K^8+4}$.
\end{prop}
\begin{remark}
Essentially the same argument shows that if $G=G_1\supseteq\ldots\supseteq G_s\supseteq G_{s+1}=\{1\}$ is the lower central series for $G$ and $G=AG_{s+1-i}$ then $G_{i+1}\subseteq A^{K^{O_i(1)}}$. We leave the details to the reader.
\end{remark}
Throughout this section, $G$ is a finitely generated nilpotent group and $A$ is a $K$-approximate subgroup such that $G=A\cdot Z(G)$, as in Proposition \ref{prop:strong.dim.lem}.

Given elements $a,b\in G$ we define, as usual, the \emph{commutator} $[a,b]$ by $[a,b]=a^{-1}b^{-1}ab$. It is well known (see \cite[\S11.1]{hall}, for example) that there exists a finite set $c_1,\ldots,c_r$ of commutators, called \emph{basic commutators}, such that the series $\{1\}=\Gamma_0\subseteq\Gamma_1\subseteq\ldots\subseteq\Gamma_r$ formed by taking $\Gamma_i=\langle c_1,\ldots,c_i\rangle$ is a central series with $[G,G]=\Gamma_r$, and such that every $x\in\Gamma_i$ can be expressed in the form $x=c_i^{\ell_i}\cdots c_1^{\ell_1}$, with $\ell_i\in\Z$ of course depending on $x$. Let these commutators $c_i$ and subgroups $\Gamma_i$ be fixed from now on.

\begin{lemma}\label{lem:coms.in.A}
The set of commutators in $G$ is contained in $A^4$.
\end{lemma}
\begin{proof}
Write $\pi:G\to G/Z(G)$ for the projection homomorphism. The commutator $[a,b]$ depends only on $\pi(a)$ and $\pi(b)$. Since $\pi(A)=G/Z(G)$, there exist $a',b'\in A$ such that $\pi(a')=\pi(a)$ and $\pi(b')=\pi(b)$, and so $[a,b]=[a',b']\in A^4$, as desired.
\end{proof}

\begin{lemma}\label{lem:A4.gens}
We have $\Gamma_i\subseteq A^4\Gamma_{i-1}$ for each $i=1,\ldots,r$.
\end{lemma}
\begin{proof}
Writing $c_i=[a_i,b_i]$, it follows from the easily verified identity $[x,yz]=[x,z]z^{-1}[x,y]z$ that $[a_i,b_i]^{\ell_i}\in[a_i,b_i^{\ell_i}]\Gamma_{i-1}$. The desired result therefore follows from Lemma \ref{lem:coms.in.A}.
\end{proof}

\begin{lemma}\label{lem:dim.lem.full.quot}
Let $j\in\{0,\ldots,r-1\}$. Then there exists $j'>j$ such that $\Gamma_{j'}\subseteq A^4\Gamma_j$, and such that either $j'=r$ or $\Gamma_{j'+1}=A^8\Gamma_j\cap\Gamma_{j'+1}\not\subseteq A^4\Gamma_j$.
\end{lemma}
\begin{proof}
Let $j'\le r$ be maximal such that $(A^4\cap\Gamma_{j'})\Gamma_j$ is a group, noting that $j'>j$ by Lemma \ref{lem:A4.gens}. Lemma \ref{lem:A4.gens} implies that $(A^4\cap\Gamma_{j'})$ generates $\Gamma_{j'}$, so in fact we have $(A^4\cap\Gamma_{j'})\Gamma_j=\Gamma_{j'}$; in particular, $\Gamma_{j'}\subseteq A^4\Gamma_j$, as required. If $j'\ne r$ then $(A^4\cap\Gamma_{j'+1})\Gamma_j$ is not a group by definition of $j'$, and in particular we have $(A^8\cap\Gamma_{j'+1})\Gamma_j\not\subseteq(A^4\cap\Gamma_{j'+1})\Gamma_j$, and hence $A^8\Gamma_j\cap\Gamma_{j'+1}\not\subseteq A^4\Gamma_j$. However, $\Gamma_{j'+1}=A^8\Gamma_j\cap\Gamma_{j'+1}$ by Lemma \ref{lem:A4.gens}.
\end{proof}
\begin{proof}[Proof of Proposition \ref{prop:strong.dim.lem}]
It follows from repeated application of Lemma \ref{lem:dim.lem.full.quot} that there exist $k\in\Z$ and $0=j(0)<j(1)<\ldots<j(k)$ such that $\Gamma_{j(i)}=A^8\Gamma_{j(i-1)}\cap\Gamma_{j(i)}\not\subseteq A^4\Gamma_{j(i-1)}$ for each $i$, and such that $[G,G]\subseteq A^4\Gamma_{j(k)}$. Lemma \ref{lem:gleason} implies that $k\le K^8$, and so we may take $H_i=\Gamma_{j(i)}$ in Proposition \ref{prop:strong.dim.lem}.
\end{proof}

\section{Nilpotent groups}\label{sec:conclusion}
In this section we prove Theorem \ref{thm:nag} under the assumption that $G$ is nilpotent. In fact, we prove the following slightly more detailed result, which includes some additional conclusions that are of use when generalising to the residually nilpotent setting.
\begin{prop}[nilpotent case of Theorem \ref{thm:nag}]\label{prop:nag}
Let $G$ be a nilpotent group and suppose that $A$ is a $K$-approximate subgroup of $G$. Then there exist subgroups $H\lhd C\subseteq G$ such that
\begin{enumerate}[label=(\alph*)]
\item\label{item:H} $H\subseteq A^{O_K(1)}$;
\item\label{item:C/H} $C/H$ is nilpotent of step at most $K^6$;
\item\label{item:Cgen} $C$ is generated by $A^6\cap C$;
\item\label{item:size'} $|A^2\cap C|\ge\exp(-K^{O(1)})|A|$;
\item\label{item:cover} $A$ can be covered by $\exp(K^{O(1)})$ left cosets of $C$.
\end{enumerate}
\end{prop}

We use the following special case of a lemma of Guralnick \cite{guralnick}.
\begin{lemma}[{\cite[Lemma 3.1]{guralnick}}]\label{lem:guralnick}
Let $G$ be a group, and let $D$ be an abelian normal subgroup of $G$ such that $G=\langle x_1,\ldots,x_n,D\rangle$. Then $[G,D]=\{\prod_{i=1}^n[x_i,d_i]:d_i\in D\}$.
\end{lemma}

\begin{proof}[Proof of Proposition \ref{prop:nag}]
Let $k,D_1,\ldots,D_{k+1},\gamma_1\ldots,\gamma_k$ be as given by Proposition \ref{prop:dim.lem}, write $Z_i=\langle\gamma_i\rangle$, and write $C=D_{k+1}$; the group $C$ acts to some extent as the ambient group in this proof. Note that Proposition \ref{prop:dim.lem} \eqref{item:dah} and \eqref{item:gam2} imply that $C$ is generated by $A^6\cap C$, and so property \ref{item:Cgen} of the present proposition is satisfied.

We have $|A^2\cap C|\ge K^{-35k}|A|$ and $k\le K^6$, and so Corollary \ref{cor:ruz.cov} implies that $A$ is covered by $K^{35K^6+2}$ translates of $C$, and properties \ref{item:size'} and \ref{item:cover} are satisfied. Moreover, for each $i$ we have $D_i$ normal in $C$, and $D_{i+1}\subseteq A^2Z_iD_i$ with $Z_i$ central in $C/D_i$. Finally, $D_1\subseteq A^2$.

We claim that there exist the following.
\begin{enumerate}[label=(\roman*)]
\item\label{item:nilp} Subgroups $\overline D_{k+1}\supseteq\overline D_k\supseteq\ldots\supseteq\overline D_1$, normal in $C$, such that $\overline D_{k+1}=C$ and $D_i\subseteq \overline D_i\subseteq A^{O_{K,(k-i)}(1)}D_i$ otherwise, and such that $\overline D_{i+1}$ is central in $C/\overline D_i$.
\item\label{item:induc} Non-negative integers $r(k)\le\ldots\le r(1)$ such that $r(i)\le O_{K,(k-i)}(1)$, and elements $x_1,\ldots,x_{r(1)}\in C$ such that $x_1,\ldots,x_{r(i)}\in A^{O_{K,(k-i)}(1)}$, and such that $\overline D_{i+1}=\langle x_{r(i+1)+1},\ldots,x_{r(i)},\overline D_i\rangle$.
\end{enumerate}
Note that part \ref{item:nilp} of the claim is enough to prove the proposition, since we may take $H=\overline D_1$, and then $(\overline D_i/\overline D_1)$ is a central series of length at most $K^6+1$ for $C/H$. Part \ref{item:induc} exists only to facilitate an inductive proof of part \ref{item:nilp}.

To prove the claim, we assume that $\overline D_{k+1},\ldots,\overline D_{i+1}$ and $x_1,\ldots,x_{r(i+1)}$ have been constructed and satisfy \ref{item:nilp} and \ref{item:induc}, and then construct $\overline D_i$ and $r(i)$. This assumption implies that there exists $m\le O_{K,(k-i)}(1)$ such that $\overline D_{i+1}=(A^m\cap\overline D_{i+1})Z_iD_i$. Lemma \ref{lem:int.app.grp} implies that  $A^m\cap\overline D_{i+1}$ is an $O_{K,(k-i)}(1)$-approximate group, and so Proposition \ref{prop:strong.dim.lem} applied to $\overline D_{i+1}/D_i$ implies that $[\overline D_{i+1},\overline D_{i+1}]\subseteq A^{O_{K,(k-i)}(1)}D_i$.

Since $\overline D_{i+1}$ is normal in $C$, so is $[\overline D_{i+1},\overline D_{i+1}]$, and so we may define a normal subgroup $D_i'=[\overline D_{i+1},\overline D_{i+1}]D_i$, noting that $D_i'\subseteq A^{O_{K,(k-i)}(1)}D_i$. Since $\overline D_{i+1}$ is abelian in $C/D_i'$, we may apply Lemma \ref{lem:guralnick} in this quotient to conclude that $[C,\overline D_{i+1}]\subseteq\{\prod_{j=1}^{r(i+1)}[x_j,d_j]:d_j\in\overline D_{i+1}\}D_i'$.

Since $Z_i$ is central in $C/D_i$, the image of $[x_j,d_j]$ in $\overline D_{i+1}/D_i$ when $d_j\in\overline D_{i+1}$ depends only on the image of $d_j$ in the quotient $\overline D_{i+1}/Z_iD_i$. Moreover, for each $d_j\in\overline D_{i+1}$ there exists $d_j'\in A^m\cap\overline D_{i+1}$ with the same image as $d_j$ in $\overline D_{i+1}/Z_iD_i$. Every such commutator $[x_j,d_j]$ therefore lies in $A^{O_{K,(k-i)}(1)}D_i$. In particular, setting $D_i''=[C,\overline D_{i+1}]D_i'$ we have $D_i''\subseteq A^{O_{K,(k-i)}(1)}D_i$, with $D_i''$ normal in $C$.

The image of $A^m\cap\overline D_{i+1}$ in $C/D_i''$ is an abelian $O_{K,(k-i)}(1)$-approximate group by Lemma \ref{lem:int.app.grp}, and so it follows from Theorem \ref{thm:ab.frei} that there exist $y_1,\ldots,y_n\in A^{O_{K,(k-i)}(1)}\cap\overline D_{i+1}$, with $n\le O_{K,(k-i)}(1)$, and a set $B\subseteq A^{O_{K,(k-i)}(1)}\cap\overline D_{i+1}$ such that $BD_i''$ is a group, such that $\overline D_{i+1}=\langle y_1,\ldots,y_n,BD_i''\rangle$. Set $r(i)=r(i+1)+n$, and set $x_{r(i+1)+j}=y_j$ for $j=1,\ldots,n$.

Since the image of $\overline D_{i+1}$ in $C/D_i''$ is central, the group $BD_i''$ is normal. Note, moreover, that $BD_i''\subseteq A^{O_{K,(k-i)}(1)}D_i$. Thus we can finally define $\overline D_i=BD_i''$, and the claim, and hence the proposition, is proved.
\end{proof}

\section{Residually nilpotent groups}\label{sec:resid}
A group $G$ is said to be \emph{residually nilpotent} if for every non-identity element $g\in G$ there exists a nilpotent group $N$ and a homomorphism $\pi:G\to N$ such that $\pi(g)\ne1$. This is a strictly weaker condition than that of being nilpotent: finitely generated free groups are residually nilpotent, for example. In this section we extend our results from nilpotent groups to this more general setting, using an argument similar to one appearing in \cite{nfdl}.

It will be convenient first to note that being residually nilpotent is in fact equivalent to an apparently slightly stronger condition, as follows.
\begin{lemma}\label{lem:strong.resid}
Let $G$ be a residually nilpotent group and let $A\subseteq G$ be a finite set such that $1\notin A$. Then there exists a nilpotent group $N$ and a homomorphism $\pi:G\to N$ such that $A\cap\ker\pi=\varnothing$.
\end{lemma}
\begin{proof}
By definition, for each $a\in A$ there exists a nilpotent group $N_a$ and a homomorphism $\pi_a:G\to N_a$ such that $\pi_a(a)\ne1$. In particualar, writing $s_a$ for the step of $N_a$ and $G=G_1\supseteq G_2\supseteq\ldots$ for the lower central series of $G$ we have $a\notin G_{s_a+1}$. Writing $s=\max_{a\in A}s_a$, we may therefore take $\pi$ to be the projection homomorphism $\pi:G\to G/G_{s+1}$.
\end{proof}

\begin{lemma}\label{lem:resid.reduc}
Let $G$ be a group, and let $A\subseteq G$ be a symmetric set containing the identity. Let $N$ be another group, and let $\pi:G\to N$ be a homomorphism. Let $H\subset\pi(A)$ be a subgroup of $N$. Then we have the following.
\begin{enumerate}
\item\label{item:2} If $A^2\cap\ker\pi=\{1\}$ then $\pi$ is injective on $A$.
\item\label{item:3} If $A^3\cap\ker\pi=\{1\}$ then there exists a subgroup $H'\subseteq A$ isomorphic to $H$ via $\pi$.
\item\label{item:4} If $A^4\cap\ker\pi=\{1\}$ then $H'$ is normal in $\langle A\rangle$ if and only if $H$ is normal in $\langle\pi(A)\rangle$.
\end{enumerate}
\end{lemma}

\begin{proof}
Item \eqref{item:2} follows from the fact that if $\pi(a)=\pi(a')$ then $a^{-1}a'\in\ker\pi$, and in turn implies that for each $h\in H$ there is a unique $\phi(h)\in A$ such that $\pi(\phi(h))=h$. Given $h,h'\in H$ we have $\phi(h)\phi(h')\phi(hh')^{-1}\in A^3\cap\ker\pi$. If $A^3\cap\ker\pi=\{1\}$, it therefore follows that $\phi(h)\phi(h')=\phi(hh')$, and hence that $H'=\phi(H)\subseteq A$ is a subgroup. Item \eqref{item:2} implies moreover that $\pi|_{H'}:H'\to H$ is an isomorphism.

If $H$ is normal in $\langle\pi(A)\rangle$ then for every $a\in A$ and $h\in H$ there exists $\hat h\in H$ such that $\pi(a^{-1})h\pi(a)=\hat h$. In particular, $a^{-1}\phi(h)a\phi(\hat h^{-1})\in\ker\pi\cap A^4$, so if $A^4\cap\ker\pi=\{1\}$ then $a^{-1}\phi(h)a\in H'$, and hence $H'$ is normal in $\langle A\rangle$.
\end{proof}

\begin{proof}[Proof of Theorem \ref{thm:pgrp}]
The theorem holds for nilpotent groups by the proof given at the end of Section \ref{sec:prelim}. Set $M=(3r+2)K^6+3$. Lemma \ref{lem:strong.resid} implies that there exists a homomorphism $\pi$ from $G$ to a nilpotent group $N$ such that $A^{3M}\cap\ker\pi=\{1\}$. Applying the nilpotent version of the theorem to $\pi(A)$, we conclude that there exists a subgroup $C\subseteq\pi(A^{M-1})$ and a subset $X\subseteq A$ with $|X|\le K^{35K^6+2}$ such that $\pi(A)\subseteq\pi(X)C$. However, Lemma \ref{lem:resid.reduc} \eqref{item:3} implies that there exists a subgroup $C'\subseteq A^{M-1}$ such that $\pi|_{C'}:C'\to C$ is an isomorphism, and so for every $a\in A$ there exist $x\in X$ and $c\in C'$ such that $\pi(a)=\pi(xc)$. Since $\pi$ is injective on $A^M$ by Lemma \ref{lem:resid.reduc} \eqref{item:2}, we conclude that $a=xc$, and so $A\subseteq XC'$ and the theorem is proved.
\end{proof}

Recall that the \emph{simple commutator of weight $k$} in the elements $x_1,\ldots,x_k$ is defined inductively by $[x_1,x_2]:=x_1^{-1}x_2^{-2}x_1x_2$ and $[x_1,\ldots,x_k]:=[[x_1,\ldots,x_{k-1}],x_k]$, and that a group $G$ generated by a set $X$ is nilpotent of step $s$ if and only if every simple commutator of weight $s+1$ in elements of $X$ is trivial \cite{hall}.

\begin{proof}[Proof of Theorem \ref{thm:nag}]
Let $m$ be the quantity implied by the $O_K(1)$ notation in Proposition \ref{prop:nag} \ref{item:H}, let $\ell$ be the word length of a simple commutator of weight $K^6+1$, and let $M\ge m(\ell+1)$. Lemma \ref{lem:strong.resid} implies that there exists a homomorphism $\pi$ from $G$ to a nilpotent group $N$ such that $A^{4M}\cap\ker\pi=\{1\}$. Applying Proposition \ref{prop:nag} to $\pi(A)$, we conclude that there exist subgroups $H\lhd C\subseteq N$ such that $H\subseteq\pi(A^m)$, such that $C/H$ is nilpotent of step at most $K^6$, such that $C$ is generated by $\pi(A^6)\cap C$, and such that $|\pi(A^2)\cap C|\ge\exp(-K^{O(1)})|\pi(A)|$.

Define $C'=\langle A^m\cap\pi^{-1}(C)\rangle$, noting that $\pi(C')=C$. Note also that $H\subseteq\pi(A^m\cap\pi^{-1}(C))=\pi(A^m\cap C')$, and so Lemma \ref{lem:resid.reduc} implies that there is a normal subgroup $H'\lhd C'$ such that $H'\subseteq A^m$ and such that $\pi|_{H'}:H'\to H$ is an isomorphism.

Set $k=K^6$. Following \cite{nfdl}, if $x_1,\ldots,x_{k+1}\in A^m\cap C'$ then the nilpotency of $C/H$ implies that $[\pi(x_1),\ldots,\pi(x_{k+1})]\in H$, which implies that there exists $h\in H'$ such that $[\pi(x_1),\ldots,\pi(x_{k+1})]\pi(h)=1$. By Lemma \ref{lem:resid.reduc} \eqref{item:2} this implies that $[x_1,\ldots,x_{k+1}]h=1$, and so we conclude that $C'/H'$ is nilpotent of step at most $K^6$.

Finally, note that $\pi(A^2)\cap C=\pi(A^2\cap C')$, and hence, by Lemma \ref{lem:resid.reduc} \eqref{item:2}, that  $|A^2\cap C'|=|\pi(A^2)\cap C|\ge\exp(-K^{O(1)})|\pi(A)|=\exp(-K^{O(1)})|A|$. The theorem therefore follows from Corollary \ref{cor:ruz.cov}.
\end{proof}

\begin{proof}[Proof of Corollary \ref{cor:nag}]
It follows from Theorem \ref{thm:nag} and Lemma \ref{lem:2.2.ii} that $A$ can be covered by $\exp(K^{O(1)})$ translates of $A^2\cap C$. Writing $\pi:C\to C/H$, Lemma \ref{lem:int.app.grp} implies that $\pi(A^2\cap C)$ is a $K^3$-approximate subgroup of the $K^6$-step nilpotent group $C/H$. The result then follows from Theorem \ref{thm:step}.
\end{proof}

\section{Growth in groups}\label{sec:growth}
\begin{proof}[Proof of Corollary \ref{cor:growth}]
We start by proving the corollary under the additional assumption that $n\ge N$, where $N\in\N$ is some constant to be determined shortly. If \eqref{eq:1-sc.gap} holds then in particular we have $|S^n|\le n^{c\log\log n}|S^{\lceil n^{1/2}\rceil}|$, which we may re-write as $|S^n|\le (\log n)^{c\log n}|S^{\lceil n^{1/2}\rceil}|=(\log n)^{2c\log n^{1/2}}|S^{\lceil n^{1/2}\rceil}|$. This implies that there exists $r\in[0,\log_5 n^{1/2}]$ such that
\[
|S^{5^{r+1}\lceil n^{1/2}\rceil}|\le (\log n)^{O(c)}|S^{5^r\lceil n^{1/2}\rceil}|.
\]
It then follows from \cite[Lemma 2.2]{bt}, for example, that $S^{2\cdot5^r\lceil n^{1/2}\rceil}$ is a $(\log n)^{O(c)}$-approximate group. Provided $c$ is small enough and $N$ is large enough, Theorem \ref{thm:nag} therefore implies that there is a subgroup $C$ of $G$ and a normal subgroup $H\lhd C$ contained in $S^{O_n(1)}$ such that $C/H$ is nilpotent of step at most $\log n-1$, and such that $S^{\lceil n^{1/2}\rceil}$ is contained in at most $n^{1/2}$ left cosets of $C$. In particular, \cite[Lemma 2.7]{bt} implies that $C$ has index at most $n^{1/2}$ in $G$.

Following the proof of \cite[Corollary 11.7]{bgt}, note that $C$ acts on $H$ by conjugation. Since $|H|\le O_n(1)$, the permutation group of $H$ has cardinality at most $O_n(1)$, and hence the stabiliser $C'<C$ of this action has index at most $O_n(1)$. This proves the corollary for $n\ge N$, since $C'$ is nilpotent of step at most $\log n$.

Replacing $c$ by a smaller constant if necessary, we may assume that $N^{c\log\log N}<2$; this ensures that if \eqref{eq:1-sc.gap} holds for some $n\in[2,N]$ then $G$ is the trivial group, and hence satisfies the  the corollary. This completes the proof for all $n\ge2$.
\end{proof}

\end{document}